\documentclass[11pt,reqno]{amsart}
\usepackage{color}

\usepackage{geometry}
\usepackage{amsmath}
\usepackage{amsfonts,amscd}
\usepackage{amssymb}
\usepackage{url}
\usepackage[english]{babel}

\usepackage{tikz}
\usepackage{forest}
\usepackage{graphicx}

\theoremstyle{plain}
\newtheorem{theorem}                 {Theorem}      [section]

\newtheorem{lemma}        [theorem]  {Lemma}

\theoremstyle{definition}

\newtheorem{definition}   [theorem]  {Definition}

\newtheorem{example}      [theorem]  {Example}

\newtheorem{remark}       [theorem]  {Remark}

\newtheorem*{acknowledge}{Acknowledgments}

\newcommand{\eq}{=}

\numberwithin{equation}{section}

\def \rn{{\mathbb R}}

\def \cn{{\mathbb C}}

\def \hn{{\mathbb H}}
\def \nn{{\mathbb N}}
\def \N{{\mathbb N}}
\def \Oct{{\mathbb O}}

\def \nab#1#2{\hbox{$\nabla$\kern -.3em\lower 1.0 ex
    \hbox{$#1$}\kern -.1 em {$#2$}}}

\def \a{\mathfrak{a}}

\def \g{\mathfrak{g}}
\def \h{\mathfrak{h}}
\def \k{\mathfrak{k}}

\def \m{\mathfrak{m}}
\def \n{\mathfrak{n}}

\def \r{\mathfrak{r}}
\def \s{\mathfrak{s}}

\def \un{\mathfrak{u}}

\DeclareMathOperator{\ad}{ad}

\DeclareMathOperator{\Aut}{Aut}

\def \SL2{\widetilde{\text{\bf SL}}_{2}(\rn)}

\numberwithin{equation}{section}

\begin{document}

\title
[$p\,$-Harmonic Functions on rank-one Lie Groups of Iwasawa type]
{Explicit $p\,$-Harmonic Functions on\\ rank-one Lie groups of Iwasawa type}

\author{Sigmundur Gudmundsson}
\address{Mathematics, Faculty of Science\\ University of Lund\\
Box 118, Lund 221\\
Sweden}
\email{Sigmundur.Gudmundsson@math.lu.se}

\author{Anna Siffert}
\address{Mathematisches Institut\\
Einsteinstr. 62\\
48149 M\" unster\\
Germany}
\email{ASiffert@uni-muenster.de}

\author{Marko Sobak}
\address{Mathematisches Institut\\
Einsteinstr. 62\\
48149 M\" unster\\
Germany}
\email{MSobak@uni-muenster.de}

\begin{abstract}
We construct explicit proper $p$-harmonic functions on rank-one Lie groups of Iwasawa type.
This class of Lie groups includes many classical Riemannian manifolds such as the rank-one symmetric spaces of non-compact type, Damek-Ricci spaces, rank-one Einstein solvmanifolds and Carnot spaces.
\end{abstract}



\maketitle

%%%%%%%%%%%%%%%%%%%%%%%%%%%%%%%%%%%%%%%%%%%%%%%%%%%%%%%%%%%%

\section{Introduction} \label{section-introduction}

Mathematicians and physicists alike have been interested in the biharmonic equation $\Delta^2\phi = 0$, where $\Delta$ is the Laplace operator, for almost two centuries. The biharmonic equation makes an appearance in fields such as continuum mechanics, elasticity theory, as well as two-dimensional hydrodynamics problems involving Stokes flows of incompressible Newtonian fluids. Biharmonic functions have quite a rich and interesting history, a survey of which can be found in the article \cite{Mel}.

\smallskip

The literature on biharmonic functions is vast, but with only very few exceptions the domains are either surfaces or open subsets of flat Euclidean space, see for example \cite{Bai-Far-Oua}. The development of
the last few years has changed this and can be traced at the regularly
updated online bibliography \cite{Gud-p-bib} maintained by the first
author. The main object of study are \textit{$p$-harmonic functions} (alt.\ \textit{polyharmonic of order $p$}) on Riemannian manifolds $(M,g)$. These are defined as solutions $\phi: (M,g) \to \cn$ of the 
$p$-harmonic equation
\begin{equation*}
\tau^p(\phi) = 0
\end{equation*}
for a positive integer $p$.
Here, $\tau$ is the Laplace-Beltrami operator on $(M,g)$ extended to complex-valued functions by requiring it to be complex linear, and $\tau^p$ denotes its $p$-ith iterate defined inductively by
\begin{equation*}
\tau^0(\phi) = \phi, \quad \tau^p(\phi) = \tau(\tau^{p-1}(\phi)).
\end{equation*}
Since each $p$-harmonic function is trivially $r$-harmonic for any $r \geq p$, one is usually interested in finding the lowest $p$ for which a function is $p$-harmonic.
Thus, one says that a function is \textit{proper $p$-harmonic} if it is $p$-harmonic but \textit{not} $(p-1)$-harmonic.
Note that in local coordinates the $p$-harmonic equation is a linear elliptic partial differential equation of order $2p$. 
This makes the construction of explicit proper $p$-harmonic functions a challenging task.

\smallskip

In this work we provide explicit examples of globally defined proper $p$-harmonic functions on rank-one Lie groups $S$ of Iwasawa type, see Definition \ref{definition-Iwasawa}. 
These Lie groups form a wide class of important Riemannian manifolds which among others includes
the rank-one symmetric spaces of non-compact type, Damek-Ricci spaces, rank-one Einstein solvmanifolds and Carnot spaces. Every Lie group of this type is diffeomorphic to its tangent space at the neutral element and hence allows for a global coordinate system $(t,x)$.  This way we get hold of an explicit expression for the corresponding Laplace-Beltrami operator, which is in general rather complicated. When restricted to a carefully chosen class of functions the Laplace-Beltrami operator simplifies in 
such a way that, for
any integer $p$, it allows the construction of a plethora of explicit proper $p$-harmonic functions on rank-one Lie groups of Iwasawa type. We thus obtain the following.

\medskip

\noindent\textbf{Main Result:} For any $p\in\N_+$ we construct proper $p$-harmonic functions $\phi_p,\psi_p : S \to \cn$ of the form
\begin{eqnarray*}
\phi_p(t,x) &=& h(x)\log(t)^{p-1}  + \text{correction terms}, \\[0.2cm]
\psi_p(t,x) &=& h(x)\, t^n \log(t)^{p-1}  + \text{correction terms},
\end{eqnarray*}
where $h$ is an appropriately chosen function independent of $t$ and
$n$ is a real number depending on the geometry of $S$.

\bigskip

\noindent \textbf{Organisation:} In Section\,\ref{section-solvable-lie-groups} we provide the necessary background on rank-one Lie groups of Iwasawa type. Afterwards we provide important examples of such Lie groups in Section\,\ref{section-special-spaces}.
In the main section, Section\,\ref{main_section}, we construct the above mentioned explicit proper $p$-harmonic functions on rank-one Lie groups of Iwasawa type. 

\begin{acknowledge}
The second and third named authors are grateful for the
funding by the Deutsche Forschungsgemeinschaft (DFG, German Research Foundation) under Germany's Excellence Strategy EXC 2044-390685587, Mathematics M\"unster: Dynamics-Geometry-Structure.
\end{acknowledge}
%%%%%%%%%%%%%%%%%%%%%%%%%%%%%%%%%%%%%%%%%%%%%%%%%%%%%%%%%%%%

\section{Rank-one Lie groups of Iwasawa type}
\label{section-solvable-lie-groups}

In this section we describe the structure of the Riemannian manifolds that we are investigating in this work.
They generalise Riemannian symmetric spaces of non-compact type, which can be presented as solvable Lie groups, via the Iwasawa decomposition of their semisimple isometry group.
\smallskip

\begin{definition}\cite{Heb}\label{definition-Iwasawa}
A {\it Lie group of Iwasawa type} is a simply connected solvable Riemannian Lie group $(S,g)$ whose Lie algebra $\s$ satisfies the following conditions:
\begin{itemize}
\item[(1)] $\s=\a\oplus\n$, where $\n=[\s,\s]$ and $\a=\n^\perp$ is abelian,
\item[(2)] $\ad(A)|_\n : \n \to \n$ is symmetric for all $A \in \a$,
\item[(3)] $\ad(A)|_\n : \n \to \n$ is positive-definite for some $A \in \a$.
\end{itemize}
The {\it rank} of a Lie group of Iwasawa type is the dimension of the subspace $\a$.
\end{definition}

For the rest of this section let us assume that $S$ is a rank-one Lie group of Iwasawa type.
Since the subspace $\a$ is one-dimensional in this case, there exists a unique element $A\in\a$ of unit length satisfying conditions (2) and (3) from Definition \ref{definition-Iwasawa}.
The spectral theorem yields the orthogonal decomposition
\begin{equation*}
\n = \bigoplus_{i=1}^m \n_i
\end{equation*}
of $\n$ into the non-trivial real eigenspaces
\begin{equation*}
\n_i = \{ X \in \n \,\mid\, \ad(A)|_\n\, X = \lambda_i X \}
\end{equation*}
of $\ad(A)|_\n$, where the eigenvalues satisfy $0 < \lambda_1 < \ldots < \lambda_m$ by the positive-definiteness of $\ad(A)|_n$.
The Lie algebra $\s=\a\oplus\n$ can therefore be viewed as the semidirect product $\a \ltimes_\pi \n$ with respect to the Lie algebra homomorphism
\begin{equation*}
\pi : \a \to \text{Der}(\n), \quad \pi(A)|_{\n_i} = \lambda_i \cdot \mathrm{id}_{\n_i}.
\end{equation*}
At the group level, this means that $S$ is the semidirect product of the standard $\rn^+$ and the unique simply connected Lie group $N$ with Lie algebra $\n$.
Thus we have $S = \rn^+ \ltimes_\mu N$ with
\begin{equation*}
\mu : \rn^+ \to \Aut(N), \quad \mu(t) \exp_\n(X) = \exp_\n\left( t^{\lambda_1} X^1 + \ldots + t^{\lambda_m} X^m \right),
\end{equation*}
where $X = X^1 + \ldots + X^m \in \n$ is the unique decomposition of $X$ such that $X^i \in \n_i$.
Written down explicitly, this means that the group law is
\begin{equation*}
(t, \exp_\n(X))(s, \exp_\n(Y)) = (ts, \; \exp_\n(X)\exp_\n(\mu(t)Y)).
\end{equation*}

\smallskip

For each of the eigenspaces $\n_i$ of $\ad(A)|_\n:\n\to\n$ we now fix an orthonormal basis
\begin{equation*}
\{X^i_{j} \in \n_i \,\mid\, j = 1,\ldots, n_i\}
\end{equation*}
of $\n$, with $n_i = \dim\n_i$.
For later purposes we also denote the structure constants of the Lie algebra $\n$ in the chosen basis by
\begin{equation*}
A^{ik\alpha}_{j\ell\beta} = \langle [X^i_j, X^k_\ell], X^\alpha_\beta \rangle.
\end{equation*}
Observe that since the operator $\ad(A)|_\n$ is a derivation, we have
\begin{equation*}
\ad(A)|_\n[X^i_j, X^k_\ell] = [\ad(A)|_\n X^i_j, X^k_\ell] + [X^i_j, \ad(A)|_\n X^k_\ell] = (\lambda_i + \lambda_k)[X^i_j, X^k_\ell].
\end{equation*}
This implies that $[X^i_j, X^k_\ell]$ is either zero, or else it must be an eigenvector of $\ad(A)|_\n$ with eigenvalue $\lambda_i+\lambda_k > \max(\lambda_i,\lambda_k)$.
Thus,
\begin{equation}\label{eq-structure-coefficients-zero}
A^{ik\alpha}_{j\ell\beta} = 0\ \ \text{if}\ \ \alpha \leq \max(i,k).
\end{equation}
In particular, this argument shows that $\n$ is at most $m$-step nilpotent.

Owing to the nilpotency of $\n$, the exponential mapping $\exp_\n : \n \to N$ is a diffeomorphism onto the simply connected Lie group $N$,
so we can obtain a natural global coordinate system $(t,x)$ on $S$ via the diffeomorphism
\begin{equation*}
S \ni (t, \exp_\n(X)) \mapsto (t,x^1_{1}, \ldots, x^1_{n_1}, \ldots, x^m_{1}, \ldots x^m_{n_m}) \in \rn^+ \times \rn^{n_1+\ldots+n_m},
\end{equation*}
where the element $X\in\n$ is uniquely decomposed as
\begin{equation*}
X = \sum_{i=1}^m \sum_{j=1}^{n_i} x^i_j X^i_j.
\end{equation*}
Throughout the paper we will also use the notation
\begin{equation*}
x^i = (x^i_1, \ldots, x^i_{n_i}) \in \rn^{n_i}, \quad i = 1,\ldots,m.
\end{equation*}

In Appendix \ref{appendix-laplace-beltrami} we give a detailed proof of the fact that the Laplace-Beltrami operator $\tau$ on $S$ is given by
\begin{equation}\label{equation-tau-formula}
\tau(\phi) = t^2\,\frac{\partial^2 \phi}{\partial t^2} + (1-n) \, t\, \frac{\partial \phi}{\partial t}
+ \sum_{i,k,\alpha=1}^m\sum_{j,\ell,\beta=1}^{n_i, n_k, n_\alpha} t^{2\lambda_i} P^{ik}_{j\ell}(x) \frac{\partial}{\partial x^k_\ell} \left( P^{i\alpha}_{j\beta}(x) \frac{\partial \phi}{\partial x^\alpha_\beta} \right).
\end{equation}
Here, $n = n_1\lambda_1 + \ldots + n_m \lambda_m$ and $P^{ik}_{j\ell}$ is a family of polynomials, of degree less than $m$, given explicitly by
\begin{equation}\label{eq-structure-polynomials}
P^{i\alpha}_{j\beta}(x)
=
\begin{cases}
\delta_{i\alpha} \, \delta_{j\beta}, & i \geq \alpha,\\[0.2cm]
\displaystyle\sum_{r=1}^{m-1}  \frac{B_r}{r!} \, p^{i\alpha}_{j\beta}(x,r), & i < \alpha,
\end{cases}
\end{equation}
where $B_r$ are the Bernoulli numbers\footnote{Here one should use the convention that $B_1 = 1/2 > 0$.} and $p^{i\alpha}_{j\beta}(x,r)$ are homogeneous polynomials of order $r$ defined inductively by
\begin{eqnarray*}
p^{i\alpha}_{j\beta}(x, 1) &=& \sum_{k=1}^{\alpha-1}\sum_{\ell=1}^{n_k} A^{ki\alpha}_{\ell j\beta} \, x^k_\ell,\\
p^{i\alpha}_{j\beta}(x, r) &=& \sum_{k=1}^{\alpha-1}\sum_{\ell=1}^{n_k} p^{i k}_{j \ell}(x, r-1)\; p^{k\alpha}_{\ell\beta}(x, 1), \quad r = 2,\ldots,m-1.
\end{eqnarray*}
Observe that the polynomials $P^{i\alpha}_{j\beta}$ depend solely on the structure coefficients of the Lie algebra $\n$.

\smallskip

It is evident that the formula for the Laplace-Beltrami operator is quite complicated in its full generality.
However, as one may expect, the formula becomes a lot simpler if we assume that the function we work with does not depend on all of the $x^i$-variables.
These simplifications are contained in the following result, which will be used in the subsequent sections to recover the formulas for the Laplace-Beltrami operator on some classical spaces, as well as for constructing explicit examples later.

\begin{lemma}\label{lemma-tau-x^1-x^2}
Let $S$ be a rank-one Lie group of Iwasawa type,
and let $\Delta_{x^i}$ denote the classical Laplacian in the $x^i$-variables
\begin{equation*}
\Delta_{x^i} = \sum_{j=1}^{n_i} \frac{\partial^2}{(\partial x^i_j)^2}, \quad i = 1,\ldots,m.
\end{equation*}
\begin{enumerate}
\item[(i)] If $h : S \to \cn$ is a function only depending on the $x^1$-variables, then
\begin{equation*}
\tau(h) = t^{2\lambda_1} \Delta_{x^1} \, h.
\end{equation*}

\vskip0.2cm

\item[(ii)] If $h : S \to \cn$ is a function depending only on the $x^1$-  and $x^2$-variables, then
\begin{eqnarray*}
\tau(h) &=&
t^{2\lambda_1} \Delta_{x^1} \, h
+ t^{2\lambda_1}
\sum_{j,\ell=1}^{n_1} \sum_{\beta=1}^{n_2} A^{112}_{j\ell\beta} \, x^1_j \, \frac{\partial^2 h}{\partial x^1_\ell \partial x^2_\beta}\\[0.1cm]
&+& t^{2\lambda_2} \Delta_{x^2} \, h
+ \frac{t^{2\lambda_1}}{4}
\sum_{j,r,s=1}^{n_1} \sum_{\ell,\beta=1}^{n_2}
A^{112}_{rj\ell} \, A^{112}_{sj\beta} \; x^1_r \, x^1_s \, \frac{\partial^2 h}{\partial x^2_\ell \partial x^2_\beta}.
\end{eqnarray*}
\end{enumerate}
\end{lemma}

\begin{proof}
See Appendix \ref{appendix-laplace-beltrami}.
\end{proof}

%%%%%%%%%%%%%%%%%%%%%%%%%%%%%%%%%%%%%%%%%%%%%%%%%%%%%%%%%%%%

\section{Examples of rank-one Lie groups of Iwasawa type}\label{section-special-spaces}

We will now briefly list some classical Riemannian manifolds which can be presented as rank-one Lie groups of Iwasawa type.

\subsection{Real Hyperbolic Spaces} Without doubt the best known examples of Riemannian manifolds satisfying the conditions of Definition \ref{definition-Iwasawa} are the real hyperbolic spaces. The space $\rn H^{n+1}$ can be presented as the simply connected Lie group $S=\rn^+ \ltimes \rn^n$ with  Lie algebra
$\s = \a \ltimes \r^n$
whose non-zero bracket relations are given by
\begin{equation*}
[A, X] = X, \quad A \in \a, \; X \in \r^n.
\end{equation*}
In this case, the nilpotent part $\n=[\s,\s]$ is the $n$-dimensional abelian Lie algebra $\r^n$ and the operator $\ad(A)|_{\r^n}$ is the identity map, so the decomposition of $\r^n$ into eigenspaces of $\ad(A)|_{\r^n}$ is trivial in the sense that the entire $\r^n$ is an eigenspace of eigenvalue $1$.  As  an immediate consequence of Lemma \ref{lemma-tau-x^1-x^2} (i) we obtain the well-known Laplace-Beltrami operator on $\rn H^{n+1}$
\begin{equation*}
\tau(\phi) = t^2\,\frac{\partial^2\phi}{\partial t^2} +(1-n)\, t\, \frac{\partial\phi}{\partial t} + t^2 \, \Delta_x\phi.
\end{equation*}
Explicit examples of proper $p$-harmonic functions on $\rn H^{n+1}$ are already known, see the recent work \cite{Gud-15}.

\subsection{Complex Hyperbolic Spaces}\label{subsection-complex-hyperbolic}
The complex hyperbolic space $\cn H^{n+1}$ can be viewed as the simply connected Lie group $S = \rn^+ \ltimes \mathrm H^{2n+1}$, where $\mathrm H^{2n+1}$ is the classical nilpotent Heisenberg group.
For simplicity let us only discuss the case $n=1$ i.e.\ the complex hyperbolic plane $S = \cn H^2$.
Its Lie algebra $\s = \a \ltimes \h^3$ is defined by the non-zero bracket relations
\begin{equation*}
[A,X] = \tfrac{1}{2}X, \quad [A,Y] = \tfrac{1}{2}Y,  \quad [A,Z] = Z, \quad  [X,Y] = Z,
\end{equation*}
where $A \in \a$ and $X,Y,Z \in \h^3$ form an orthonormal basis for the Heisenberg algebra $\h^3$.
In this case the operator $\ad(A)|_{\h^3}$ has two eigenvalues, namely $1/2$ and $1$,
and the nilpotent part decomposes as $\h^3 = \n_1 \oplus \n_2$ with $X,Y \in \n_1$ and $Z \in \n_2$. 
Note that, according to this definition, the sectional curvatures of $\cn H^2$ lie between $-1$ and $-1/4$, and these extremal values are attained in the planes spanned by $A,Z$ and $A,X$ (or $A,Y$), respectively.

As described in Section \ref{section-solvable-lie-groups}, we obtain a global coordinate system on $\cn H^2$,
where we denote the coordinates by
\begin{equation*}
(t,x,y,z) = (t,x^1_1,x^1_2,x^2_1)
\end{equation*}
for simplicity.
Now as a consequence of Lemma \ref{lemma-tau-x^1-x^2} (ii), we see that the Laplace-Beltrami operator on $\cn H^2$ satisfies
\begin{eqnarray*}
\tau(\phi)
&=&
t^2\frac{\partial^2\phi}{\partial t^2} - t \, \frac{\partial\phi}{\partial t}
+ 
t^{2}
\left(\frac{\partial^2\phi}{\partial x^2} + \frac{\partial^2\phi}{\partial y^2} \right)\\[0.2cm]
&+& \frac{t^2(x^2+y^2)+4t^4}{4}\,\frac{\partial^2\phi}{\partial z^2} 
+t^2\left(
x \frac{\partial^2\phi}{\partial y\partial z}
- y \frac{\partial^2\phi}{\partial x\partial z}
\right).
\end{eqnarray*}

\subsection{Homogeneous Spaces of Negative Curvature}\label{subsection-symmetric}

In his paper \cite{Hei}, E. Heintze proves that any Riemannian homogeneous space of negative curvature can be presented as a solvable Lie group $S$ such that the codimension of the derived algebra $[\s,\s]$ has codimension one in $\s$.
In particular, each homogeneous space of negative cuvature satisfies condition (1) of Definition \ref{definition-Iwasawa}. Furthermore, he shows that there exists an element $A \in \a$ such that the symmetric part of $\ad(A)|_\n$ is positive-definite. Thus, a homogeneous space of negative curvature is of Iwasawa type if and only if $\ad(A)|_\n$ is symmetric for all $A \in \a$. As particular examples of such spaces we have the rank-one symmetric spaces of non-compact type i.e.\  the already mentioned real and complex hyperbolic spaces $\rn H^n$ and $\cn H^n$, but also the quaternionic hyperbolic spaces $\hn H^n$ and the $16$-dimensional hyperbolic Cayley plane $\Oct H^2$.

\subsection{Damek-Ricci Spaces}\label{subsection-damek-ricci}

Another class of much studied Riemannian manifolds is that of Damek-Ricci spaces. They generalize the hyperbolic spaces $\cn H^n,\hn H^n, \Oct H^2$ and can be presented as rank-one Lie groups of Iwasawa type. These spaces are solvable Lie groups of the form $S= \rn^+ \ltimes N$, where $N$ is a generalized Heisenberg group.
The operator $\ad(A)|_\n$ on a Damek-Ricci space by definition has eigenvalues $\lambda_1 = 1/2$ and $\lambda_2 = 1$,
and the eigenspace corresponding to the eigenvalue $1$ is the center of $\n$.
For more details on these spaces we refer to \cite{BTV}.

\subsection{Einstein Solvmanifolds}\label{subsection-einstein}
We recall that a Riemannian manifold $(M,g)$ is said to be an \textit{Einstein manifold} if its Ricci tensor is a constant multiple of the metric i.e.\
\begin{equation*}
\mathrm{Ric}_g = cg, \quad c \in \rn.
\end{equation*}
Recall further that a \textit{solvmanifold} is a solvable Lie group $S$ endowed with a left invariant Riemannian metric.
The Lie algebra $\mathfrak{s}$ of any Einstein solvmanifold can be orthogonally decomposed as
$\s = \a \oplus \n$ where $\n = [\s,\s]$ is nilpotent and $\a = \n^\perp$ is abelian -- see \cite{Heb} and \cite{Lau}.
In particular, each Einstein solvmanifold satisfies condition (1) of Definition \ref{definition-Iwasawa}.
Furthermore, if one assumes that $S$ is non-flat, then $S$ is in fact of Iwasawa type, and the eigenvalues of $\ad(A)|_\n$ are positive integers without a common divisor \cite{Heb}. A comprehensive survey on Einstein solvmanifolds can be found in \cite{Lau}.

%%%%%%%%%%%%%%%%%%%%%%%%%%%%%%%%%%%%%%%%%%%%%%%%%%%%%%%%%%%%

\section{Explicit proper $p$-harmonic functions}\label{main_section}
In this main section we provide explicit proper $p$-harmonic functions on some Riemannian manifolds. Subsection\,\ref{sub_main} contains the construction of explicit complex-valued proper $p$-harmonic functions on rank-one Lie groups of Iwasawa type.
In Subsection\,\ref{section-duality} we briefly discuss how these solutions induce
$p$-harmonic functions on symmetric spaces of compact type.

\subsection{$p$-Harmonic Functions on Rank-One Lie groups of Iwasawa Type}\label{sub_main}
In this subsection we construct explicit complex-valued proper $p$-harmonic functions on rank-one Lie groups of Iwasawa type. Before we state our main result, let us present concrete examples in two special cases.

\begin{example}\label{example-biharmonic-rh2-ch2}
\leavevmode
\begin{enumerate} 
\item[(i)] The function $\phi : \rn H^2 \to \rn$ given by
\begin{eqnarray*}
\phi(t,x)
&=&
x^{6} \log\left(t\right)
- 15 \, x^4 t^{2} {\left(\log\left(t\right) - 2\right)}\\[0.1cm]
&+&  5 \, x^2 t^{4} {\left(3  \log\left(t\right) - 8\right)} 
-\frac{t^6}{15} \,{\left(15  \log\left(t\right) - 46\right)}
\end{eqnarray*}
is proper biharmonic on $\rn H^2$.

\vskip0.2cm

\item[(ii)] The function $\psi : \cn H^2 \to \rn$ given by
\begin{eqnarray*}
\psi(t,x,y,z)
&=&
z^4t^2\log(t) 
- \frac{(x^2+y^2)z^2t^3}{3}\,(3\log(t) - 20)\\[0.15cm]
&-& \frac{z^2t^4}{4}\,(2\log(t)-1) 
+ \frac{((x^2+y^2)^2 + 8z^2)\,t^4}{96}\,(6\log(t)-70) \\[0.2cm]
&+& \frac{(x^2+y^2)\,t^5}{300} \, (30\log(t)-17)
+ \frac{t^6}{75}\,(5\log(t)-2)
\end{eqnarray*}
is proper biharmonic on $\cn H^2$.
\end{enumerate}
\end{example}

Motivated by these particular examples,
we seek proper $p$-harmonic functions $\phi_p,\psi_p : S \to \cn$ of the form
\begin{eqnarray*}
\phi_p(t,x) &=& h(x)\log(t)^{p-1}  + \text{correction terms}, \\[0.2cm]
\psi_p(t,x) &=& h(x)\, t^n \log(t)^{p-1}  + \text{correction terms},
\end{eqnarray*}
where $h$ is an appropriately chosen function independent of $t$ and we recall that 
$n = n_1\lambda_1 + \ldots + n_m\lambda_m.$
Observe that, since $h$ is independent of $t$, the formula (\ref{equation-tau-formula}) for the Laplace-Beltrami operator implies that there exist functions $h_1,\ldots,h_m$ independent of $t$ such that
\begin{equation*}
\tau(h) = \sum_{k=1}^m h_i(x) \, t^{2\lambda_i}.
\end{equation*}
To be able to state which functions $h$ are suitable for us, we need to fix some notations and definitions first.

For non-negative integers $r,m$, we use the multi-index set notation
\begin{equation*}
I_r^m = \{ \alpha \in \nn^r \,\mid\, 1 \leq \alpha_1,\ldots,\alpha_r \leq m \}.
\end{equation*}

\begin{definition}
Let $S$ be a rank-one Lie group of Iwasawa type and let $h : S \to \cn$ be a function independent of $t$.
Then the \textit{tension tree} of $h$ is the unique collection
\begin{equation*}
\{ h^i_\alpha : S \to \cn \,\mid\, i \in \nn, \ \alpha \in I_i^m \}
\end{equation*}
of functions independent of $t$ such that
\begin{eqnarray*}
\tau(h) &=& \sum_{k=1}^m h^1_k(x) \, t^{2\lambda_k}\\[0.1cm]
\tau(h^i_\alpha) &=& \sum_{k=1}^m h^{i+1}_{(\alpha,k)}(x) \, t^{2\lambda_k},
\quad i \in \nn, \ \alpha \in I_i^m.
\end{eqnarray*}
If there exists a number $r \in \nn$ such that $h^{r+1}_\alpha \equiv 0$ for all $\alpha \in I_{r+1}^m$, then the smallest such number $r$ is said to be the \textit{degree} of the tension tree of $h$, and otherwise the degree is said to be infinite.
The elements $h^i_\alpha$ of the tension tree of $h$ are called \textit{nodes}.
For a fixed $\alpha \in I_r^m$, the sequence
\begin{equation*}
\{ h^i_{(\alpha_1,\ldots,\alpha_i)} \,\mid\, i = 1,\ldots,r\}
\end{equation*}
is called a \textit{branch} of length $r$.
\end{definition}

\begin{example}\label{example-tension-trees}
\leavevmode
\begin{enumerate}
\item[(i)] For a function $h : \rn H^2 \to \cn$ independent of $t$ we have
\begin{equation*}
h^1_1(x) = \frac{\partial^2 h}{\partial x^2}, 
\qquad
h^{i+1}_{(1,\ldots,1)}(x) = \frac{\partial^2 h^i_{(1,\ldots,1)}}{\partial x^2}.
\end{equation*}
E.g.\ the tension tree of the function $h(x) = x^6$ is of degree $3$, with non-zero nodes
\begin{equation*}
h^1_1(x) = 30\,x^4, \qquad h^2_{(1,1)} = 360\, x^2, \qquad h^3_{(1,1,1)} = 720.
\end{equation*}

\vskip0.2cm

\item[(ii)] For a less trivial example, consider a function $h : \cn H^2 \to \cn$ independent of $t$.
Its tension tree satisfies the relations
\begin{eqnarray*} 
h^1_1(x,y,z)
&=&
\frac{\partial^2 h}{\partial x^2} + \frac{\partial^2 h}{\partial y^2}
+ x \,\frac{\partial^2 h}{\partial y\partial z}
- y \,\frac{\partial^2 h}{\partial x\partial z}
+ \frac{x^2+y^2}{4} \, \frac{\partial^2 h}{\partial z^2},\\[0.2cm]
h^1_2(x,y,z)
&=&
\frac{\partial^2 h}{\partial z^2},
\end{eqnarray*}
and analogous formulae hold for $h^i_\alpha$ with $i > 1$.
For example, the tension tree of $z^4$ on $\cn H^2$ is displayed in Figure \ref{fig-tension-tree}.
\end{enumerate}
\end{example}

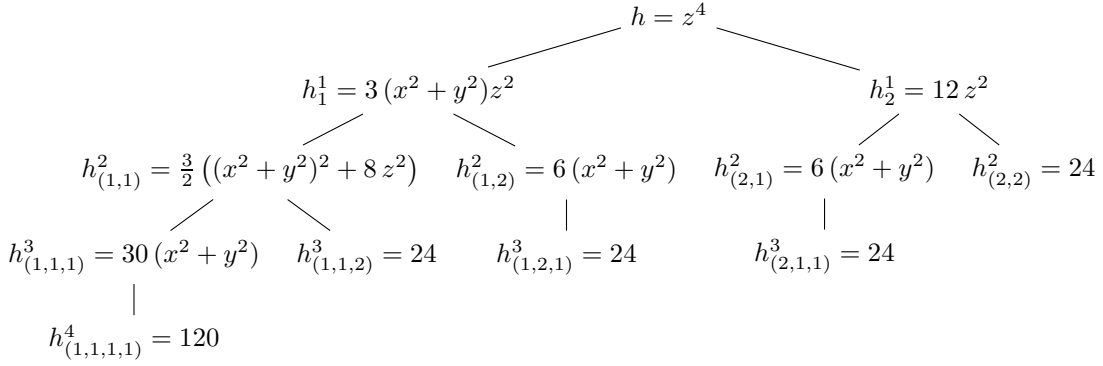
\begin{figure}[t]
\centering

\makebox[\textwidth][c]{
\small
\begin{forest}
[$h \eq z^{4} $ 
	[$h^1_{1} \eq 3 \, (x^{2}+y^2) z^{2} $ 
		[$h^2_{(1, 1)} \eq  \frac{3}{2}\left((x^{2} + y^{2})^2 + 8 \, z^{2}\right) $ 
			[$h^3_{(1, 1, 1)} \eq 30 \, (x^{2} + y^{2}) $ 
				[$h^4_{(1, 1, 1, 1)} \eq 120 $ 
				] 
			] 
			[$h^3_{(1, 1, 2)} \eq 24 $ 
			] 
		] 
		[$h^2_{(1, 2)} \eq 6 \, (x^{2} + y^{2}) $ 
			[$h^3_{(1, 2, 1)} \eq 24 $ 
			] 
		] 
	] 
	[$h^1_{2} \eq 12 \, z^{2} $ 
		[$h^2_{(2, 1)} \eq 6 \, (x^{2} + y^{2}) $ 
			[$h^3_{(2, 1, 1)} \eq 24 $ 
			] i
		] 
		[$h^2_{(2, 2)} \eq 24 $ 
		] 
	] 
]
\end{forest}
}
\caption{The tension tree of $z^4$ on $\cn H^2$.}
\label{fig-tension-tree}
\end{figure}

The observant reader may have noticed a certain resemblance between the biharmonic functions from Example \ref{example-biharmonic-rh2-ch2} 
and the tension trees of the corresponding functions $h$ discussed in Example \ref{example-tension-trees}. 
The following result is a far-reaching generalization of this idea.

\begin{theorem}\label{theorem-p-harmonic-main}
Let $S$ be a rank-one Lie group of Iwasawa type.
Let $h : S \to \cn$ be a function independent of $t$ and suppose that its tension tree $\{h^i_\alpha\}$ is of finite degree $r$.
For $i \in \nn$ and $\alpha \in I_i^m$ put
\begin{equation*}
\Lambda^k_\alpha = \sum_{j=1}^k \lambda_{\alpha_j}, \quad k=1,\ldots,i.
\end{equation*}
Further, for $p,i \in \nn$ and $\alpha \in I_i^m$, define
\begin{eqnarray*}
f^i_{\alpha} (t,p) &=& \frac{t^{2\Lambda^i_\alpha}}{\prod_{k=1}^i \Lambda^k_\alpha} \sum_{j=0}^{p-1}
\sum_{\ell_1 + \ldots + \ell_i = j}
\frac{(-1)^{i+j} \,2^{j-i} \prod_{k=1}^j (p-k)}{\prod_{k=1}^i \left(2\Lambda^k_\alpha-n\right)^{\ell_k+1}}  \log(t)^{p-j-1},\\[0.2cm]
g^i_{\alpha} (t,p) &=& \frac{t^{2\Lambda^i_\alpha+n}}{\prod_{k=1}^i \Lambda^k_\alpha} \sum_{j=0}^{p-1}
\sum_{\ell_1 + \ldots + \ell_i = j}
\frac{(-1)^{i+j} \,2^{j-i} \prod_{k=1}^j (p-k)}{\prod_{k=1}^i \left(2\Lambda^k_\alpha+n\right)^{\ell_k+1}}  \log(t)^{p-j-1},
\end{eqnarray*}
where we additionally assume that
\begin{equation*}
\prod_{k=1}^i (2\Lambda^k_\alpha-n) \not= 0 \quad\text{if}\quad h^i_\alpha \not= 0, \quad\text{for}\quad i \in \nn, \ \alpha \in I_i^m
\end{equation*}
in the definition of $f^i_\alpha$.
Then the functions $\phi_p,\psi_p : S \to \cn$ given by
\begin{eqnarray*}
\phi_p(t,x) &=& h(x)\log(t)^{p-1} + \sum_{i=1}^{r} \sum_{\alpha \in I_i^m} h^i_\alpha(x) \, f^i_{\alpha}(t,p),\\[0.1cm]
\psi_p(t,x) &=& h(x)\,t^n\log(t)^{p-1} + \sum_{i=1}^{r} \sum_{\alpha \in I_i^m} h^i_\alpha(x) \, g^i_{\alpha}(t,p)
\end{eqnarray*}
are proper $p$-harmonic on $S$.
Furthermore, for any non-zero $(a,b) \in \cn^2$ the linear combination $a\phi_p+b\psi_p$ is proper $p$-harmonic on $S$.
\end{theorem}

\begin{remark}
Note that the assumption $\prod_{k=1}^i (2\Lambda^k_\alpha-n) \not= 0$ is needed to ensure that $\phi_p$ is well-defined.
On the other hand, $\psi_p$ is well-defined even without this assumption.
\end{remark}

\begin{remark}
Taking $h$ to be harmonic, we see from Theorem \ref{theorem-p-harmonic-main} that the function
\begin{equation*}
(t,x) \mapsto (a + b \, t^n) \log(t)^{p-1} \, h(x)
\end{equation*}
is proper $p$-harmonic.
A result of this form has already been obtained in \cite{Gud-15} for the case $S = \rn H^n$,
and in \cite{Gud-Sob-2} for the cases $S = \rn^m \ltimes \rn^n$ and $S = \rn^m \ltimes \mathrm H^{2n+1}$, where $\mathrm H^{2n+1}$ denotes the classical $(2n+1)$-dimensional Heisenberg group.
\end{remark}

Before proving Theorem\,\ref{theorem-p-harmonic-main}, we provide a preparatory lemma.

\begin{lemma}\label{lem_prep}
Let $i \geq 1$ and $j \geq 0$ be integers and let $a \in \cn^i$. Then	
\begin{equation*}
\sum_{\ell_1+\ldots+\ell_i = j} \, \prod_{k=1}^{i-1} a_k^{\ell_k+1} a_i^{\ell_i}
-
\sum_{\ell_1+\ldots+\ell_i+1 = j} \, \prod_{k=1}^i a_k^{\ell_k+1}
=
\sum_{\ell_1+\ldots+\ell_{i-1} = j} \, \prod_{k=1}^{i-1} a_k^{\ell_k+1}
\end{equation*}
\end{lemma}

\begin{proof}
Write
\begin{equation*}
\sum_{\ell_1+\ldots+\ell_i = j} \, \prod_{k=1}^{i-1} a_k^{\ell_k+1} a_i^{\ell_i}
=
\sum_{\ell_i = 0}^j a_i^{\ell_i} \sum_{\ell_1+\ldots+\ell_{i-1} = j-\ell_i} \, \prod_{k=1}^{i-1} a_k^{\ell_k+1}
\end{equation*}
and
\begin{eqnarray*}
\sum_{\ell_1+\ldots+\ell_i+1 = j} \, \prod_{k=1}^i a_k^{\ell_k+1}
&=&
\sum_{\ell_i = 0}^{j-1} a_i^{\ell_i+1} \sum_{\ell_1+\ldots+\ell_{i-1} = j-(\ell_i+1)} \,  \prod_{k=1}^{i-1} a_k^{\ell_k+1}\\
&=&
\sum_{\ell_i = 1}^{j} a_i^{\ell_i} \sum_{\ell_1+\ldots+\ell_{i-1} = j-\ell_i} \,  \prod_{k=1}^{i-1} a_k^{\ell_k+1},
\end{eqnarray*}
where we perform the variable change $\ell_i+1 \mapsto \ell_i$.
Subtracting the latter two, we see that only the term with $\ell_i = 0$ remains and the claim follows.
\end{proof}

\begin{proof}[Proof of Theorem \ref{theorem-p-harmonic-main}]
Let us first focus on the function $\phi_p$.
A direct calculation shows that $\phi_1$ is proper harmonic and $\phi_2$ is proper biharmonic.
The result then follows immediately by induction from the fact that
\begin{equation*}
\tau(\phi_p) = -n(p-1)\phi_{p-1} + (p-1)(p-2)\phi_{p-2},
\end{equation*}
so it only remains to prove this identity.
We begin by calculating
\begin{eqnarray}\label{eq-solution-generalized-identity-proof-0}
&& \tau(h(x)\log(t)^{p-1}) \\[0.2cm]
&=&
\sum_{k=1}^m h^1_k(x) \,t^{2\lambda_k}\log(t)^{p-1} + h(x)\left(n(1-p)\log(t)^{p-2}
+(p-1)(p-2)\log(t)^{p-3}\right). \nonumber
\end{eqnarray}
Now
\begin{equation}\label{eq-solution-generalized-identity-proof-1}
\tau(h^i_\alpha(x) \, f^i_{\alpha}(t,p))
=
\sum_{k=1}^m h^{i+1}_{(\alpha,k)}(x) \, t^{2\lambda_k} f^i_\alpha(t,p)
+
h^i_\alpha(x) \, \tau(f^i_{\alpha}(t,p)),
\end{equation}
so we wish to calculate
\begin{equation*}
\tau(f^i_{\alpha}(t,p))
=
\frac{1}{\prod_{k=1}^i \Lambda^k_\alpha} \sum_{j=0}^{p-1}
\sum_{\ell_1 + \ldots + \ell_i = j}
\frac{(-1)^{i+j} \,2^{j-i} \prod_{k=1}^j (p-k)}{\prod_{k=1}^i \left(2\Lambda^k_\alpha-n\right)^{\ell_k+1}}  \tau(t^{2\Lambda^i_\alpha}\log(t)^{p-j-1}).
\end{equation*}
We have
\begin{eqnarray*}
&&\tau(t^{2\Lambda_\alpha^i}\log(t)^{p-j-1})\\[0.2cm]
&=& t^{2\Lambda_\alpha^i} \Big( 2\Lambda_\alpha^i(2\Lambda_\alpha^i-n)\log(t)^{p-j-1} + 4\Lambda_\alpha^i(p-j-1)\log(t)^{p-j-2} \\
&& \quad - n(p-j-1)\log(t)^{p-j-2} + (p-j-1)(p-j-2)\log(t)^{p-j-3} \Big).
\end{eqnarray*}
Hence,
\begin{eqnarray}\label{eq-solution-generalized-identity-proof-2}
&&\tau(f^i_\alpha(t,p)) \nonumber\\
&=&
\frac{t^{2\Lambda_\alpha^i}}{\prod_{k=1}^{i-1} \Lambda_\alpha^k} \sum_{j=0}^{p-1}\; \sum_{\ell_1+\ldots+\ell_i = j} \frac{(-1)^{i+j} \,2^{j-i+1} \prod_{k=1}^j(p-k)}{\prod_{k=1}^{i-1}(2\Lambda_\alpha^k-n)^{\ell_k+1} \, (2\Lambda_\alpha^i-n)^{\ell_i}} \,\log(t)^{p-j-1} \nonumber\\[0.1cm]
&+& \frac{t^{2\Lambda_\alpha^i}}{\prod_{k=1}^{i-1} \Lambda_\alpha^k} \sum_{j=0}^{p-1}\; \sum_{\ell_1+\ldots+\ell_i = j} \frac{(-1)^{i+j} \,2^{j-i+2} \prod_{k=1}^{j+1}(p-k)}{\prod_{k=1}^i(2\Lambda_\alpha^k-n)^{\ell_k+1}} \,\log(t)^{p-j-2} \nonumber\\[0.1cm]
&-& \frac{nt^{2\Lambda_\alpha^i}}{\prod_{k=1}^i \Lambda_\alpha^k} \sum_{j=0}^{p-1}\; \sum_{\ell_1+\ldots+\ell_i = j} \frac{(-1)^{i+j} \, 2^{j-i} \prod_{k=1}^{j+1}(p-k)}{\prod_{k=1}^i(2\Lambda_\alpha^k-n)^{\ell_k+1}} \,\log(t)^{p-j-2} \nonumber\\[0.1cm]
&+& \frac{t^{2\Lambda_\alpha^i}}{\prod_{k=1}^i \Lambda_\alpha^k} \sum_{j=0}^{p-1}\; \sum_{\ell_1+\ldots+\ell_i = j} \frac{(-1)^{i+j} \, 2^{j-i} \prod_{k=1}^{j+2}(p-k)}{\prod_{k=1}^i(2\Lambda_\alpha^k-n)^{\ell_k+1}} \,\log(t)^{p-j-3}.
\end{eqnarray}
Note that the second term in (\ref{eq-solution-generalized-identity-proof-2}) can be rewritten as
\begin{eqnarray*}
&& \frac{t^{2\Lambda_\alpha^i}}{\prod_{k=1}^{i-1} \Lambda_\alpha^k} \sum_{j=0}^{p-1}\; \sum_{\ell_1+\ldots+\ell_i = j} \frac{(-1)^{i+j} \,2^{j-i+2} \prod_{k=1}^{j+1}(p-k)}{\prod_{k=1}^i(2\Lambda_\alpha^k-n)^{\ell_k+1}} \,\log(t)^{p-j-2} \\[0.1cm]
&=&
-\frac{t^{2\Lambda_\alpha^i}}{\prod_{k=1}^{i-1} \Lambda_\alpha^k} \sum_{j=1}^{p-1}\; \sum_{\ell_1+\ldots+\ell_i+1 = j} \frac{(-1)^{i+j} \,2^{j-i+1} \prod_{k=1}^{j}(p-k)}{\prod_{k=1}^i(2\Lambda_\alpha^k-n)^{\ell_k+1}} \,\log(t)^{p-j-1}
\end{eqnarray*}
as seen by making the variable change $j \mapsto j-1$ and using the fact that the term corresponding to $j = p-1$ in the left-hand side is zero.
From Lemma\,\ref{lem_prep} we have
\begin{eqnarray*}
&& \sum_{\ell_1+\ldots+\ell_i = j} \frac{1}{\prod_{k=1}^{i-1}(2\Lambda_\alpha^k-n)^{\ell_k+1} \, (2\Lambda_\alpha^i-n)^{\ell_i}}
-
\sum_{\ell_1+\ldots+\ell_i+1 = j} \frac{1}{\prod_{k=1}^{i}(2\Lambda_\alpha^k-n)^{\ell_k+1}}\\[0.1cm]
&=& \sum_{\ell_1+\ldots+\ell_{i-1} = j} \frac{1}{\prod_{k=1}^{i-1}(2\Lambda_\alpha^k-n)^{\ell_k+1}},
\end{eqnarray*}
where it is understood that if $i = 1$, then the summation is 0 unless $j=0$, in which case the sum is 1.
This implies that the first two terms in (\ref{eq-solution-generalized-identity-proof-2}) contribute the term
\begin{equation*}
\frac{t^{2\Lambda_\alpha^i}}{\prod_{k=1}^{i-1} \Lambda_\alpha^k} \sum_{j=0}^{p-1}\; \sum_{\ell_1+\ldots+\ell_{i-1} = j} \frac{(-1)^{i+j} \,2^{j-i+1} \prod_{k=1}^j(p-k)}{\prod_{k=1}^{i-1}(2\Lambda_\alpha^k-n)^{\ell_k+1}} \,\log(t)^{p-j-1}
=
-t^{2\lambda_{\alpha_i}} f^{i-1}_{\alpha'}(t,p),
\end{equation*}
where $\alpha = (\alpha',\alpha_i)$ and it is understood that
\begin{equation*}
f^{0}_{\alpha'}(t,p) = \log(t)^{p-1}.
\end{equation*}
The third term in (\ref{eq-solution-generalized-identity-proof-2}) can be rewritten as
\begin{equation*}
-\frac{nt^{2\Lambda_\alpha^i}}{\prod_{k=1}^i \Lambda_\alpha^k} \sum_{j=0}^{p-1}\; \sum_{\ell_1+\ldots+\ell_i = j} \frac{(-1)^{i+j} \, 2^{j-i} \prod_{k=1}^{j+1}(p-k)}{\prod_{k=1}^i(2\Lambda_\alpha^k-n)^{\ell_k+1}} \,\log(t)^{p-j-2}
=
n (1-p) f^i_\alpha(t,p-1),
\end{equation*}
as seen by performing the variable change $k \mapsto k+1$ in the product in the numerator and noticing that the term corresponding to $j=p-1$ in the sum is zero.
One similarly simplifies the fourth term in (\ref{eq-solution-generalized-identity-proof-2}) and obtains
\begin{equation*}
\tau(f^i_\alpha(t,p))  = -t^{2\lambda_{\alpha_i}} f^{i-1}_{\alpha'}(t,p) + n (1-p)f^i_\alpha(t,p-1) + (p-1)(p-2) f^i_\alpha(t,p-2).
\end{equation*}
Combining this identity with (\ref{eq-solution-generalized-identity-proof-0}) and (\ref{eq-solution-generalized-identity-proof-1}), we get
\begin{eqnarray*}
&& \tau(\phi_p)\\[0.2cm]
&=&
\sum_{k=1}^m h^1_k(x) \,t^{2\lambda_k}\log(t)^{p-1}
+ h(x)\left(n(1-p)\log(t)^{p-2}
+(p-1)(p-2)\log(t)^{p-3}\right)\\[0.1cm]
&+& \sum_{i=1}^{r} \sum_{\alpha \in I_i^m} \left\{ \sum_{k=1}^m h^{i+1}_{(\alpha,k)}(x) \, t^{2\lambda_k} f^i_\alpha(t,p) \right.\\[0.2cm]
&& \quad \left. +
h^i_\alpha(x) \Big( -t^{2\lambda_{\alpha_i}} f^{i-1}_{\alpha'}(t,p) + n (1-p)f^i_\alpha(t,p-1) + (p-1)(p-2) f^i_\alpha(t,p-2) \Big) \right\}\\[0.2cm]
&=&
n(1-p)\phi_{p-1} + (p-1)(p-2)\phi_{p-2} +  \sum_{k=1}^m h^1_k(x) \,t^{2\lambda_k}\log(t)^{p-1}\\[0.2cm]
&+& \sum_{i=1}^{r-1} \sum_{\alpha\in I_i^m} \sum_{k=1}^m h^{i+1}_{(\alpha,k)}(x) \, t^{2\lambda_k} f^i_\alpha(t,p)
- \sum_{i=1}^{r}  \sum_{\alpha\in I_i^m} h^i_\alpha(x) \, t^{2\lambda_{\alpha_i}} f^{i-1}_{\alpha'}(t,p)\\[0.2cm]
&=& n(1-p)\phi_{p-2} + (p-1)(p-2)\phi_{p-3},
\end{eqnarray*}
where we employ the fact that $h^{r+1}_\alpha \equiv 0$ as well as
\begin{eqnarray*}
&& \sum_{i=1}^{r}  \sum_{\alpha\in I_i^m} h^i_\alpha(x) \, t^{2\lambda_{\alpha_i}} f^{i-1}_{\alpha'}(t,p)\\[0.2cm]
&=&
\sum_{\alpha \in I_1^m} h^1_{\alpha}(x) \, t^{2\lambda_{\alpha_1}} f^{0}_{\alpha'}(t,p)
+
\sum_{i=2}^{r}  \sum_{\alpha'\in I_{i-1}^m}\sum_{\alpha_i =1}^m h^{i}_{(\alpha',\alpha_i)}(x) \, t^{2\lambda_{\alpha_i}} f^{i-1}_{\alpha'}(t,p)\\[0.2cm]
&=&
\sum_{k=1}^m h^1_k(x) \, t^{2\lambda_k} \log(t)^{p-1}
+
\sum_{i=1}^{r-1}\sum_{\alpha \in I_i^m} \sum_{k=1}^m h^{i+1}_{(\alpha,k)} \, t^{2\lambda_k} f^i_\alpha(t,p).
\end{eqnarray*}
This finishes the proof of the proper $p$-harmonicity of $\phi_p$.

To prove that $\psi_p$ is proper $p$-harmonic, one performs a similar argument to show that
\begin{equation*}
\tau(\psi_p) = n(p-1)\psi_{p-1} + (p-1)(p-2)\psi_{p-2}.
\end{equation*}
Finally, the fact that $a\phi_p+b\psi_p$ is proper $p$-harmonic for $(a,b) \in \cn^2\setminus\{0\}$ follows easily from the identities for $\tau(\phi_p)$ and $\tau(\psi_p)$.
\end{proof}

With Theorem \ref{theorem-p-harmonic-main} at hand, we now switch our focus to finding examples of functions $h$ independent of $t$ whose tension tree is of finite degree.
We have already seen that certain polynomials on the hyperbolic planes $\rn H^2$ and $\cn H^2$ have tension trees of finite degree.
This is no coincidence, and in fact we have this general result.

\begin{theorem}\label{theorem-polynomial-tension-tree}
Let $S$ be a rank-one Lie group of Iwasawa type.
Then the tension tree of any polynomial independent of $t$ is of finite degree.
\end{theorem}

Before we start with the proof of Theorem \ref{theorem-polynomial-tension-tree}, we make some crucial observations.
\begin{enumerate}
\item[(i)]
If $h,u:S\to\cn$ depend only on the $x^1,\ldots,x^\gamma$-variables for $1\leq\gamma\leq m$, then the tension field $\tau$ and the conformality operator\footnote{For two complex-valued functions $f,h:(M,g)\to\cn$ we have the
following well-known relation
$\label{equation-basic}
\tau(fh)=\tau(f)\,h+2\,\kappa(f,h)+f\,\tau(h)$,
where the {\it conformality} operator $\kappa$ is given by
$
\kappa(f,h)=g(\nabla f,\nabla h).
$} $\kappa$ satisfy
\begin{equation*}
\tau(h)
= \sum_{i=1}^\gamma t^{2\lambda_i} \left(
\sum_{k,\alpha=1}^\gamma \sum_{\ell,\beta=1}^{n_k,n_\alpha}
R^{ik\alpha}_{\ell\beta}(x) \frac{\partial^2h}{\partial x^k_\ell \partial x^\alpha_\beta}
+\sum_{\alpha=1}^\gamma \sum_{\beta=1}^{n_\alpha} Q^{i\alpha}_{\beta}(x)\frac{\partial h}{\partial x^\alpha_\beta} \right),\\[0.1cm]
\end{equation*}
\begin{equation*}
\kappa(h,u)
=
\sum_{i=1}^\gamma t^{2\lambda_i}
\left( \sum_{k,\alpha=1}^\gamma \sum_{\ell,\beta=1}^{n_k,n_\alpha}
R^{ik\alpha}_{\ell\beta}(x) \frac{\partial h}{\partial x^k_\ell}\frac{\partial u}{x^\alpha_\beta} \right),
\end{equation*}
where $R^{ik\alpha}_{\ell\beta}(x)$ and $Q^{i\alpha}_\beta(x)$ are polynomials depending only on the $x^1,\ldots,x^{\gamma-1}$-variables.

\vskip0.2cm

\item[(ii)] If $h,u : S \to\cn$ are independent of $t$, then the tension tree of their sum is the sum of their tension trees i.e.\
\begin{equation*}
(h+u)^i_\alpha = h^i_\alpha + u^i_\alpha, \quad i \in \nn, \ \alpha \in I_i^m.
\end{equation*}
In particular, if $h$ and $u$ are of finite degree $r$, then so is $h+u$.
\end{enumerate}

\begin{proof}[Proof of Theorem \ref{theorem-polynomial-tension-tree}]
The proof is by induction.
First, if $h:S\to\cn$ depends only on $x^1$, then by Lemma \ref{lemma-tau-x^1-x^2} (i) the tension tree of $h$ is given by
\begin{equation*}
h^i_{(1,1,\ldots,1)}(x^1) = \Delta_{x^1}^i h,
\end{equation*}
which is zero for large enough $i$ if $h$ is assumed to be a polynomial.
Consequently, the tension tree is of finite degree in this case.
Now for the induction step, suppose that the tension tree of any polynomial depending only on the $x^1,\ldots,x^{\gamma-1}$-variables is of finite degree and let $h$ be an arbitrary polynomial in the $x^1,\ldots,x^{\gamma}$ variables.
Keeping track only of the degree in the $x^\gamma$-variables, we can write such a polynomial as
\begin{equation*}
h(x) = \sum_{\beta'} \sum_{|\beta_\gamma| \leq N} c_\beta \cdot (x')^{\beta'} \cdot (x^\gamma)^{\beta_\gamma},
\end{equation*}
where we use the multi-multiindex notations
\begin{equation*}
\beta = (\beta',\beta_\gamma) = (\beta_1,\ldots, \beta_\gamma) \in \nn^{n_1}\times\ldots\times\nn^{n_\gamma},
\end{equation*}
\begin{equation*}
x' = (x^1,\ldots, x^{\gamma-1}), \qquad (x')^{\beta'} = \prod_{k=1}^{\gamma-1} (x^k)^{\beta_k}.
\end{equation*}
Let $(u_{\beta'})_\alpha^i$ be the tension tree of $u_{\beta'}(x') = (x')^{\beta'}$, which is of finite degree by the induction hypothesis.
By the product rule
\begin{equation*}
\tau(h) = \sum_{\beta'} \sum_{|\beta_\gamma| \leq N} c_\beta \left(
\tau[(x')^{\beta'}] \cdot (x^\gamma)^{\beta_\gamma} + 2\kappa[(x')^{\beta'},(x^\gamma)^{\beta_\gamma}]
+ (x')^{\beta'}\cdot \tau[(x^\gamma)^{\beta_\gamma}] \right).
\end{equation*}
By (i) of the discussion preceding the proof we see that, in the last two terms, the degree in the $x^\gamma$-variables decreases by at least 1, while the degree in the other variables may increase.
Then (ii) of the discussion preceding the proof implies that the first level of the tension tree of $h$ satisfies
\begin{equation*}
h^1_k(x) = \sum_{\beta'} \sum_{|\beta_\gamma| \leq N} c_\beta \cdot (u_{\beta'})^1_{k}(x') \cdot (x^\gamma)^{\beta_\gamma}
+ \sum_{\beta'} \sum_{|\beta_\gamma| \leq N-1} (c_\beta)^1_k \cdot (x')^{\beta'} \cdot (x^\gamma)^{\beta_\gamma}
\end{equation*}
for appropriately chosen coefficients $(c_\beta)^1_k$.
Iterating this argument, we get
\begin{equation*}
h^i_\alpha(x) =  \sum_{\beta'} \sum_{|\beta_\gamma| \leq N} c_\beta \cdot (u_{\beta'})^i_{\alpha}(x') \cdot (x^\gamma)^{\beta_\gamma} 
+
\sum_{\beta'} \sum_{|\beta_\gamma| \leq N-1} (c_\beta)^i_{\alpha} \cdot (x')^{\beta'} \cdot (x^\gamma)^{\beta_\gamma},
\end{equation*}
for $i \in \nn, \,\alpha \in I_i^m$, and appropriately chosen coefficients $(c_\beta)^i_{\alpha}$.
But since the tension tree of $u_{\beta'}$ is of finite degree, the first term eventually vanishes i.e.\ there is a positive integer $r$ such that $(u_{\beta'})^r_\alpha = 0$, so that
\begin{equation*}
h^r_\alpha(x)
=
\sum_{\beta'} \sum_{|\beta_\gamma| \leq N-1} (c_\beta)^r_{\alpha} \cdot (x')^{\beta'} \cdot (x^\gamma)^{\beta_\gamma},
\quad
\alpha \in I_r^m.
\end{equation*}
Thus, the degree in the $x^\gamma$-variables eventually decreases.
Applying the same argument to each node $h^r_\alpha$ (at most) $N-1$ times, we see that, eventually, we will completely lose dependence of $x^\gamma$ i.e.\ there is a positive integer $s$ such that
\begin{equation*}
h^s_\alpha(x) = \sum_{\beta'} (c_{\beta'})^s_{\alpha} \cdot (x')^{\beta'}, \quad \alpha \in I_s^m.
\end{equation*}
Finally, each node $h^s_\alpha$ is a polynomial depending only on the $x^1,\ldots,x^{\gamma-1}$ variables
and hence has tension tree of finite degree by the induction hypothesis.
This implies that the tension tree of $h$ has finite degree as well.
\end{proof}

A combination of Theorems \ref{theorem-p-harmonic-main} and \ref{theorem-polynomial-tension-tree} generates a plethora of explicit examples of proper $p$-harmonic functions on an arbitrary rank-one Lie group of Iwasawa type.
Nevertheless, it is natural to ask whether there exist non-polynomial functions independent of $t$ whose tension tree is of finite degree. This question is answered to some extend in the following example.

\begin{example}\label{example-h-tree-single-branch}
Let $S$ be a rank-one Lie group of Iwasawa type, and let $r \in \nn$ be fixed.

\begin{enumerate}
\item[(i)]
Let $H : \rn^{n_1} \to \cn$ be proper $(r+1)$-harmonic on its flat Euclidean domain and
define $h:S\to\cn$ by
\begin{equation*}
h(x) = H(x^1).
\end{equation*}
Then by Lemma \ref{lemma-tau-x^1-x^2} (i), the tension tree of $h$ has a single branch with non-zero nodes
\begin{equation*}
h^i_{(1,1,\ldots,1)}(x) = \Delta^i H(x^1), \quad i= 1,\ldots, r,
\end{equation*}
and the tree has degree $r$.

\vskip0.2cm

\item[(ii)]
Let $H : \rn^{n_1} \setminus \{0\} \to \cn$ be the radial proper $(r+1)$-harmonic function
\begin{equation*}
H(x^1) =
\begin{cases}
\sum_{k=0}^{r} \Vert x^1\Vert^{2k} (a_k \log \Vert x^1\Vert + b_k), & n_1 = 2 \\[0.2cm]
\sum_{k=0}^{r} \Vert x^1\Vert^{2k} (a_k \Vert x^1\Vert^{2-n_1} + b_k), & n_1 \not= 2,
\end{cases}
\end{equation*}
where $a,b \in \cn^{r+1}$ are non-zero,
and define the function $G : \rn^{n_2} \to \cn$ by
\begin{equation*}
G(x^2) = c_0 + \sum_{j=1}^{n_2} c_j\,x^2_j
\end{equation*}
for a non-zero $c \in \cn^{n_2+1}$.
Define $h:S\to\cn$ by
\begin{equation*}
h(x) = H(x^1)\,G(x^2).
\end{equation*}
Then by Lemma \ref{lemma-tau-x^1-x^2} (ii), the tension tree of $h$ has a single branch and the non-zero nodes are
\begin{equation*}
h^i_{(1,1,\ldots,1)}(x) = \Delta^iH(x^1)\,G(x^2), \quad i = 1,\ldots,r,
\end{equation*}
and the tree has degree $r$.
\end{enumerate}
\end{example}

In particular, if $n_1\geq 2$ then the examples above generate non-polynomial functions whose tension tree is of finite degree. These examples of functions only depend on the $x^1$- and the $x^2$-variables,
which seems to be the best that one can obtain in full generality.
However, one can easily construct Lie groups for which there exist non-polynomial functions whose tension tree is of finite degree and which depend on the $x^i$-variables with $i > 2$.

%%%%%%%%%%%%%%%%%%%%%%%%%%%%%%%%%%%%%%%%%%%%%%%%%%%%%%%%%%%%

\subsection{$p$-Harmonic Functions on Symmetric Spaces of Compact Type}\label{section-duality}
It is a well-known consequence of the Baker-Campbell-Hausdorff formula that every Lie group is a real-analytic manifold.  The solutions constructed in this work are all real-analytic, in particular, those from the Riemannian symmetric spaces $\rn H^n$, $\cn H^n$, $\hn H^n$ and $\Oct H^2$ of non-compact type.

Now present one of those as the non-compact quotient $G/K$ where $G$ is the connected component of the isometry group containing the neutral element and $K$ the isotropy group.  Then we have the Cartan decomposition $\g=\k\oplus\m$ of the Lie algebra $\g$ of $G$ and $\m$ its orthogonal complement.  Let $f:G/K\to\cn$ be one of our global proper $p$-harmonic solutions and $\hat f=f\circ\pi:G\to\cn$ be the $K$-invariant composition of $f$ with the natural projection $\pi:G\to G/K$.  Then extend $\hat f$ to a holomorphic function $\hat f^*:W^*\to\cn$ locally defined on the complexification $G^\cn$ of $G$.  The complex Lie group $G^\cn$ contains the compact subgroup $U$ with Lie algebra $\un=\k\oplus i\m$.  Let $\hat f^*:W^*\cap U\to\cn$ be the $K$-invariant restriction of $\hat f^*$ to $W^*\cap U$.  Then this induces a function $f^*:\pi^*(W^*\cap U)\to\cn$ locally defined of the symmetric space $U/K$ which is the compact companion of $G/K$.  Here $\pi^*:U\to U/K$ is the corresponding natural projection.

\begin{theorem}\label{theorem-duality}
Let $f:G/K\to\cn$ be a complex-valued proper $p$-harmonic function, defined on the non-compact $G/K$, as above.  Then its dual function $f^*:\pi^*(W^*\cap U)\to\cn$, locally defined on the compact $U/K$, is proper $p$-harmonic.
\end{theorem}

\begin{proof}  The result is a direct consequence of the duality principle introduced for harmonic morphisms in Theorem 7.1 of \cite{Gud-Sve-1} and then developed further for $p$-harmonic functions in Theorem 8.1 of \cite{Gud-Mon-Rat-1}.
\end{proof}

To see how this principle works explicitly, in the real hyperbolic case, we refer the reader to Theorem 5.3 of \cite{Gud-15}.

%%%%%%%%%%%%%%%%%%%%%%%%%%%%%%%%%%%%%%%%%%%%%%%%%%%%%%%%%%%%

%\section{Acknowledgements}

\appendix

%%%%%%%%%%%%%%%%%%%%%%%%%%%%%%%%%%%%%%%%%%%%%%%%%%%%%%%%%%%%

\section{The Laplace-Beltrami Operator}
\label{appendix-laplace-beltrami}

This appendix can be seen as our engine room.  Here we prove the global formula
\begin{equation}\label{equation-tau-formula-appendix}
\tau(\phi) = t^2\,\frac{\partial^2 \phi}{\partial t^2} + (1-n) \, t\, \frac{\partial \phi}{\partial t}
+ \sum_{i,k,\alpha=1}^m\sum_{j,\ell,\beta=1}^{n_i, n_k, n_\alpha} t^{2\lambda_i} P^{ik}_{j\ell}(x) \frac{\partial}{\partial x^k_\ell} \left( P^{i\alpha}_{j\beta}(x) \frac{\partial \phi}{\partial x^\alpha_\beta} \right)
\end{equation}
for the Laplace-Beltrami operator $\tau$ on the rank-one Lie groups of Iwasawa type.
Here and throughout the rest of this appendix we use the notation introduced in Section  \ref{section-solvable-lie-groups}.
We recall that, by definition, the Laplace-Beltrami operator on the Lie group $S$ satisfies
\begin{equation}\label{eq-tau-definition}
\tau=A^2-\nabla_AA+\sum_{i=1}^m\sum_{j=1}^{n_i}\bigl((X^i_j)^2-\nabla_{X^i_j}{X^i_j}\bigr),
\end{equation}
since the left-invariant vector fields $A, X^i_j\in \s$ form an orthonormal basis for the Lie algebra $\s$.
Our main goal is therefore to obtain global formulae for these left-invariant vector fields.

\medskip

We begin by noting that the Baker-Campbell-Hausdorff formula implies that the group law on $N$ is given by
\begin{equation*}
\exp_\n(X) \exp_\n(Y) = \exp_\n(X*Y),
\end{equation*}
where
\begin{eqnarray}\label{eq-group-law}
X*Y &=& X + \frac{\ad(X)}{1 - e^{-\ad(X)}} Y + O(Y^2) \nonumber\\[0.1cm]
&=& X + \sum_{r=0}^{m-1} \frac{B_r}{r!}\ad(X)^r Y + O(Y^2),
\end{eqnarray}
and where $B_r$ are the Bernoulli numbers\footnote{Here one should use the convention $B_1 = 1/2 > 0$.}.
Note that the sum in (\ref{eq-group-law}) stops at $r=m-1$ since $\n$ is $m$-step nilpotent.

The following technical result will be needed later to calculate the coordinate expressions of the left-invariant vector fields $A, X^i_j \in \s$. 

\begin{lemma}\label{lemma-group-law-expansion}
Let
\begin{equation*}
X = \sum_{k=1}^m \sum_{l=1}^{n_k} x^k_l X^k_l
\end{equation*}
be an arbitrary element of $\n$.

\begin{enumerate}
\item[(i)] For any integer $r \geq 1$, we have
\begin{equation*}
\ad(X)^r X^i_j = \sum_{\alpha = 1}^m \sum_{\beta = 1}^{n_\alpha} p^{i\alpha}_{j\beta}(x, r) \,X^\alpha_\beta
\end{equation*}
where
\begin{eqnarray*}
p^{i\alpha}_{j\beta}(x, 1) &=& \sum_{k=1}^{\alpha-1}\sum_{\ell=1}^{n_k} A^{ki\alpha}_{\ell j\beta} \, x^k_\ell,\\
p^{i\alpha}_{j\beta}(x, r) &=& \sum_{k=1}^{\alpha-1}\sum_{\ell=1}^{n_k} p^{i k}_{j \ell}(x, r-1)\cdot p^{k\alpha}_{\ell\beta}(x, 1), \quad r\geq 2.
\end{eqnarray*}
In particular, each $p^{i\alpha}_{j\beta}(x,r)$ is either identically zero or a homogeneous polynomial in $x$ of degree $r$.

\vspace{0.2cm}

\item[(ii)] For any $s \in \rn$,
\begin{equation}\label{eq-lemma-group-law-formula}
X * (s t^{\lambda_i} X^i_j) = X + s t^{\lambda_i} \sum_{\alpha = 1}^m \sum_{\beta = 1}^{n_\alpha} P^{i\alpha}_{j\beta}(x) \, X^\alpha_\beta + R(s)
\end{equation}
where $R$ denotes the remaining terms satisfying $\dot R(0) = 0$ and
\begin{equation*}
P^{i\alpha}_{j\beta}(x)
=
\begin{cases}
\delta_{i\alpha}\,\delta_{j\beta}, & i \geq \alpha\\[0.2cm]
\displaystyle\sum_{r=1}^{m-1}  \frac{B_r}{r!}\cdot p^{i\alpha}_{j\beta}(x,r), & i < \alpha.
\end{cases}
\end{equation*}
are polynomials of degree less or equal than $m-1$.
\end{enumerate}
\end{lemma}

\begin{proof}
For $r=1$, part (i) follows by a straightforward calculation in which one makes use of (\ref{eq-structure-coefficients-zero}).
The claim for $r\geq 2$ then follows by induction and (\ref{eq-structure-coefficients-zero}). 

\smallskip

For part (ii), we first note that the group law (\ref{eq-group-law}) and part (i) imply that
the formula (\ref{eq-lemma-group-law-formula}) holds with
\begin{equation*}
P^{i\alpha}_{j\beta}(x) = \delta_{i\alpha}\,\delta_{j\beta} + \sum_{r=1}^{m-1}  \frac{B_r}{r!}\cdot p^{i\alpha}_{j\beta}(x,r).
\end{equation*}
To obtain the desired formula, we note that if $i \geq \alpha$ then
\begin{equation*}
p^{i\alpha}_{j\beta}(x,1) = \sum_{k=1}^{\alpha-1}\sum_{\ell=1}^{n_k} A^{ki\alpha}_{\ell j\beta} \, x^k_\ell = 0
\end{equation*}
since $A^{ki\alpha}_{j\ell\beta} = 0$ whenever $\alpha \leq \max(k,i)$.
Assuming that $p^{i\alpha}_{j\beta}(x,r) = 0$ for $i \geq \alpha$, we get
\begin{equation*}
p^{i\alpha}_{j\beta}(x,r+1) = \sum_{k=i+1}^{\alpha-1}\sum_{\ell=1}^{n_k} p^{i k}_{j \ell}(x, r)\cdot p^{k\alpha}_{\ell\beta}(x, 1) = 0
\end{equation*}
since the summation over $k$ is empty whenever $i \geq \alpha$.
Hence, induction shows that $p^{i\alpha}_{j\beta}(x,r) = 0$ for all $r \geq 1$ if $i \geq \alpha$.
This also gives the desired formula for $i \geq \alpha$.
The formula for the case $i < \alpha$ is trivial.
\end{proof}

\medskip

\begin{lemma}\label{lemma-vector-fields}
The left-invariant vector fields $A, X^i_j \in \s$ satisfy
\begin{equation*}
A = t\,\frac{\partial}{\partial t},
\qquad
X^i_j = t^{\lambda_i}\sum_{\alpha = 1}^m \sum_{\beta = 1}^{n_\alpha} P^{i\alpha}_{j\beta}(x) \frac{\partial}{\partial x^\alpha_\beta},
\qquad
\nabla_AA = 0,
\qquad
\nabla_{X^i_j}X^i_j = \lambda_i\,t\,\frac{\partial}{\partial t}.
\end{equation*}
\end{lemma}

\begin{proof}
Let $p = (t, \exp_\n(X))$ denote an arbitrary point on the Lie group $S$, where
\begin{equation*}
X = \sum_{i=1}^m \sum_{j=1}^{n_i} x^i_j X^i_j.
\end{equation*}
Let $L_p:S\to S$ denote the left-translation by the general element $p\in S$.
Then the left-invariant vector field $A:S\to TS$ satisfies
\begin{eqnarray*}
A(p)&=&(dL_p)_e(A(e))\\[0.1cm]
&=& (dL_p)_e \Bigl( \frac{d}{ds}(e^s,\exp_\n(0))\Big|_{s=0} \Bigr)\\
&=&\frac{d}{ds}\Bigl((t, \exp_\n(X))(e^s,\exp_\n(0))\Bigr)\Big|_{s=0}\\
&=&\frac{d}{ds}(te^s, \exp_\n(X))\Big|_{s=0}\\
&=&t\frac{\partial}{\partial t}.
\end{eqnarray*}
Furthermore
\begin{eqnarray*}
X^i_j(p) &=& (dL_p)_e(X^i_j(e))\\[0.1cm]
&=& (dL_p)_e \Bigl( \frac{d}{ds}(1,\exp_\n(sX^i_j))\Big|_{s=0} \Bigr)\\
&=&\frac{d}{ds}\Bigl((t, \exp_\n(X))(s \, t^{\lambda_i} X^i_j)\Bigr)\Big|_{s=0}\\
&=& \frac{d}{ds}(t, \exp_\n (X * (s \, t^{\lambda_i} X^i_j))\Big|_{s=0}\\[0.1cm]
&=& \dot\gamma^i_j(0),
\end{eqnarray*}
where the curve $\gamma^i_j : \rn \to S$ is defined by
\begin{equation*}
\gamma^i_j(s) = (t, \exp_\n (X * (s \, t^{\lambda_i} X^i_j)).
\end{equation*}
It now follows by Lemma \ref{lemma-group-law-expansion} that
\begin{equation*}
\frac{d(t\circ \gamma^i_j)}{ds}\Big|_0 = 0, \quad
\frac{d(x^\alpha_\beta\circ \gamma^i_j)}{ds}\Big|_0 =  t^{\lambda_i} P^{i\alpha}_{j\beta}(x),
\end{equation*}
and hence
\begin{equation*}
X^i_j = t^{\lambda_i}\sum_{\alpha=1}^m\sum_{\beta=1}^{n_\alpha} P^{i\alpha}_{j\beta}(x) \frac{\partial}{\partial x^\alpha_\beta}.
\end{equation*}

To obtain the claimed formulae for the covariant derivatives, we take an arbitrary element $s A + X \in \s$ with $X \in \n$ and apply the Koszul formula to obtain
\begin{equation*}
\langle \nabla_A A, sA+X \rangle = -\langle[sA+X, A],A \rangle = 0,
\end{equation*}
since $\a = [\s,\s]^\perp$.
This shows that the integral curves of the vector field $A\in\a$ are geodesics.
On the other hand
\begin{eqnarray*}
\langle \nabla_{X^i_j}X^i_j, sA+X \rangle
&=& \langle [sA+ X, X^i_j], X^i_j \rangle\\
&=& s \langle \ad(A)X^i_j, X^i_j \rangle + \sum_{\alpha=1}^m\sum_{\beta=1}^{n_\alpha} x^\alpha_\beta \langle [X^\alpha_\beta, X^i_j], X^i_j \rangle\\
&=& s\lambda_i + \sum_{\alpha=1}^m\sum_{\beta=1}^{n_\alpha} A^{\alpha ii}_{\beta jj} x^\alpha_\beta\\
&=&s\lambda_i,
\end{eqnarray*}
since $A^{\alpha ii}_{\beta jj} =0$ by (\ref{eq-structure-coefficients-zero}).
This gives
$$\langle\nabla_{X^i_j}{X^i_j},A\rangle=\lambda_i
\ \ \text{and}\ \
\langle\nabla_{X^i_j}{X^i_j},X^\alpha_\beta\rangle=0,$$
and the result follows.
\end{proof}

After the above technical preparations, the desired formula (\ref{equation-tau-formula-appendix})
for the Laplace-Beltami operator $\tau$ on $S$ follows by inserting the identities from Lemma \ref{lemma-vector-fields} into the general formula (\ref{eq-tau-definition}).

\medskip

Finally, we provide a proof of Lemma \ref{lemma-tau-x^1-x^2}, i.e.\ we show how the formula for the Laplace-Beltrami operator simplifies when we apply it to a function depending only on the $x^1$- and the $x^2$-variables.

\begin{proof}[Proof of Lemma \ref{lemma-tau-x^1-x^2}]
To prove statement (i), note that (\ref{eq-structure-polynomials}) implies that
\begin{equation*}
P^{i1}_{j\ell}(x) = \delta_{i1}\,\delta_{j\ell}.
\end{equation*}
Since $h$ depends only on the $x^1$-variables, we therefore get
\begin{eqnarray*}
\tau(h)
&=& \sum_{i,k,\alpha=1}^m\sum_{j,\ell,\beta=1}^{n_i, n_k, n_\alpha} t^{2\lambda_i} P^{ik}_{j\ell}(x) \frac{\partial}{\partial x^k_\ell} \left( P^{i\alpha}_{j\beta}(x) \frac{\partial h}{\partial x^\alpha_\beta} \right)\\[0.2cm]
&=& \sum_{i,k=1}^m\sum_{j,\ell,\beta=1}^{n_i, n_k, n_1} t^{2\lambda_i} P^{ik}_{j\ell}(x) \frac{\partial}{\partial x^k_\ell} \left( P^{i1}_{j\beta}(x) \frac{\partial h}{\partial x^1_\beta} \right)\\[0.2cm]
&=& t^{2\lambda_1}\sum_{k=1}^m \sum_{j,\ell=1}^{n_1, n_k}
P^{1k}_{j\ell}(x) \frac{\partial^2h}{\partial x^k_\ell \, \partial x^1_j}\\
&=& t^{2\lambda_1} \sum_{j,\ell=1}^{n_1}
P^{11}_{j\ell}(x) \frac{\partial^2h}{\partial x^1_\ell \, \partial x^1_j}
=
t^{2\lambda_1} \Delta_{x^1} \, h.
\end{eqnarray*}

To prove statement (ii), we first observe that
\begin{equation*}
P^{12}_{j\ell}(x) = \frac{1}{2}\sum_{r=1}^{n_1} A^{112}_{rj\ell} \, x^1_r,
\end{equation*}
as is easily seen from (\ref{eq-structure-polynomials}).
Since $h$ depends only on the $x^1$- and $x^2$-variables,
we see that
\begin{eqnarray}\label{eq-tau-lemma-2}
\tau(h)
&=& \sum_{i,k,\alpha=1}^m\sum_{j,\ell,\beta=1}^{n_i, n_k, n_\alpha} t^{2\lambda_i} P^{ik}_{j\ell}(x) \frac{\partial}{\partial x^k_\ell} \left( P^{i\alpha}_{j\beta}(x) \frac{\partial h}{\partial x^\alpha_\beta} \right) \nonumber\\[0.2cm]
&=& \sum_{i,k=1}^m\sum_{j,\ell,\beta=1}^{n_i, n_k, n_1} t^{2\lambda_i} P^{ik}_{j\ell}(x) \frac{\partial}{\partial x^k_\ell} \left( P^{i1}_{j\beta}(x) \frac{\partial h}{\partial x^1_\beta} \right)\nonumber\\[0.2cm]
&+&
\sum_{i,k=1}^m\sum_{j,\ell,\beta=1}^{n_i, n_k, n_2} t^{2\lambda_i} P^{ik}_{j\ell}(x) \frac{\partial}{\partial x^k_\ell} \left( P^{i2}_{j\beta}(x) \frac{\partial h}{\partial x^2_\beta} \right)
\end{eqnarray}
The first sum in (\ref{eq-tau-lemma-2}) becomes
\begin{eqnarray*}
&& \sum_{i,k=1}^m\sum_{j,\ell,\beta=1}^{n_1, n_k, n_\alpha} t^{2\lambda_i} P^{ik}_{j\ell}(x) \frac{\partial}{\partial x^k_\ell} \left( \delta_{i1} \,\delta_{j\beta} \, \frac{\partial h}{\partial x^1_\beta} \right)\\[0.2cm]
&=&
t^{2\lambda_1}\sum_{k=1}^m\sum_{j,\ell=1}^{n_1, n_k}
P^{1k}_{j\ell}(x) \frac{\partial^2 h}{\partial x^k_\ell \,\partial x^1_j}\\[0.2cm]
&=&
t^{2\lambda_1} \sum_{j,\ell=1}^{n_1}
P^{11}_{j\ell}(x) \frac{\partial^2 h}{\partial x^1_\ell \,\partial x^1_j}
+
t^{2\lambda_1} \sum_{j,\ell=1}^{n_1, n_2}
P^{12}_{j\ell}(x) \frac{\partial^2 h}{\partial x^2_\ell \,\partial x^1_j}\\[0.2cm]
&=&
t^{2\lambda_1} \,\Delta_{x^1} \,h
+
\frac{t^{2\lambda_1}}{2}
\sum_{j,\ell=1}^{n_1} \sum_{\beta=1}^{n_2} A^{112}_{j\ell\beta} \, x^1_j \, \frac{\partial^2 h}{\partial x^1_\ell \partial x^2_\beta},
\end{eqnarray*}
where we rename the indices in the final term.
The second sum in (\ref{eq-tau-lemma-2}) satisfies
\begin{eqnarray*}
&& \sum_{i,k=1}^m\sum_{j,\ell,\beta=1}^{n_i, n_k, n_2} t^{2\lambda_i} P^{ik}_{j\ell}(x) \frac{\partial}{\partial x^k_\ell} \left( P^{i2}_{j\beta}(x) \frac{\partial h}{\partial x^2_\beta} \right)\\[0.2cm]
&=&
t^{2\lambda_1} \sum_{k=1}^m\sum_{j,\ell,\beta=1}^{n_1, n_k, n_2} P^{1k}_{j\ell}(x) \frac{\partial}{\partial x^k_\ell} \left( P^{12}_{j\beta}(x) \frac{\partial h}{\partial x^2_\beta} \right)\\[0.2cm]
&+&
t^{2\lambda_2} \sum_{k=1}^m\sum_{j,\ell,\beta=1}^{n_2, n_k, n_2}  P^{2k}_{j\ell}(x) \frac{\partial}{\partial x^k_\ell} \left( P^{22}_{j\beta}(x) \frac{\partial h}{\partial x^2_\beta} \right).
\end{eqnarray*}
Similarly as in statement (i), one can show that the second term in the equation above satisfies
\begin{equation*}
t^{2\lambda_2} \sum_{k=1}^m\sum_{j,\ell,\beta=1}^{n_2, n_k, n_2}  P^{2k}_{j\ell}(x) \frac{\partial}{\partial x^k_\ell} \left( P^{22}_{j\beta}(x) \frac{\partial h}{\partial x^2_\beta} \right)
=
t^{2\lambda_2} \Delta_{x^2}\,h.
\end{equation*}
On the other hand,
\begin{eqnarray*}
&& t^{2\lambda_1} \sum_{k=1}^m\sum_{j,\ell,\beta=1}^{n_1, n_k, n_2} P^{1k}_{j\ell}(x) \frac{\partial}{\partial x^k_\ell} \left( P^{12}_{j\beta}(x) \frac{\partial h}{\partial x^2_\beta} \right)\\[0.2cm]
&=& t^{2\lambda_1} \sum_{k=1}^m\sum_{j,\ell,\beta=1}^{n_1, n_k, n_2} P^{1k}_{j\ell}(x) \frac{\partial P^{12}_{j\beta}}{\partial x^k_\ell} \frac{\partial h}{\partial x^2_\beta}
+
t^{2\lambda_1} \sum_{k=1}^m\sum_{j,\ell,\beta=1}^{n_1, n_k, n_2} P^{1k}_{j\ell}(x)
P^{12}_{j\beta}(x) \frac{\partial^2 h}{\partial x^k_\ell \, \partial x^2_\beta}.
\end{eqnarray*}
Now we observe that
\begin{equation*}
P^{1k}_{j\ell}(x) \frac{\partial P^{12}_{j\beta}}{\partial x^k_\ell}
=
P^{1k}_{j\ell}(x) \cdot A^{k12}_{\ell j\beta}
=
P^{11}_{j\ell}(x) \cdot A^{112}_{\ell j\beta}
= A^{112}_{jj\beta} = 0
\end{equation*}
where we use the skew-symmetry $A^{112}_{rj\ell} = -A^{112}_{jr\ell}$
as well as the fact that $A^{k12}_{\ell j\beta} = 0$ when $k \geq 2$, cf.\ (\ref{eq-structure-coefficients-zero}).
Finally,
\begin{eqnarray*}
&& t^{2\lambda_1} \sum_{k=1}^m\sum_{j,\ell,\beta=1}^{n_1, n_k, n_2} P^{1k}_{j\ell}(x)
P^{12}_{j\beta}(x) \frac{\partial^2 h}{\partial x^k_\ell \, \partial x^2_\beta}\\[0.2cm]
&=& t^{2\lambda_1} \sum_{j,\ell,\beta=1}^{n_1, n_1, n_2} P^{11}_{j\ell}(x)
P^{12}_{j\beta}(x) \frac{\partial^2 h}{\partial x^1_\ell \, \partial x^2_\beta}
+
t^{2\lambda_1} \sum_{j,\ell,\beta=1}^{n_1, n_2, n_2} P^{12}_{j\ell}(x)
P^{12}_{j\beta}(x) \frac{\partial^2 h}{\partial x^2_\ell \, \partial x^2_\beta}\\[0.2cm]
&=&
\frac{t^{2\lambda_1}}{2}
\sum_{j,\ell=1}^{n_1} \sum_{\beta=1}^{n_2} A^{112}_{j\ell\beta} \, x^1_j \, \frac{\partial^2 h}{\partial x^1_\ell \partial x^2_\beta}
+
\frac{t^{2\lambda_1}}{4}
\sum_{j,r,s=1}^{n_1} \sum_{\ell,\beta=1}^{n_2}
A^{112}_{rj\ell} \, A^{112}_{sj\beta} \; x^1_r \, x^1_s \, \frac{\partial^2 h}{\partial x^2_\ell \partial x^2_\beta}.
\end{eqnarray*}
Combining the calculations above, we see the second sum in (\ref{eq-tau-lemma-2}) satisfies
\begin{eqnarray*}
&& \sum_{i,k=1}^m\sum_{j,\ell,\beta=1}^{n_i, n_k, n_2} t^{2\lambda_i} P^{ik}_{j\ell}(x) \frac{\partial}{\partial x^k_\ell} \left( P^{i2}_{j\beta}(x) \frac{\partial h}{\partial x^2_\beta} \right)\\[0.2cm]
&=& t^{2\lambda_2}\Delta_{x^2} \, h
+
\frac{t^{2\lambda_1}}{2}
\sum_{j,\ell=1}^{n_1} \sum_{\beta=1}^{n_2} A^{112}_{j\ell\beta} \, x^1_j \, \frac{\partial^2 h}{\partial x^1_\ell \partial x^2_\beta}\\[0.2cm]
&+&
\frac{t^{2\lambda_1}}{4}
\sum_{j,r,s=1}^{n_1} \sum_{\ell,\beta=1}^{n_2}
A^{112}_{rj\ell} \, A^{112}_{sj\beta} \; x^1_r \, x^1_s \, \frac{\partial^2 h}{\partial x^2_\ell \partial x^2_\beta},
\end{eqnarray*}
and the result follows.
\end{proof}

\end{document}